\documentclass[11pt]{amsart}

\newtheorem{theorem}{Theorem}[section]
\theoremstyle{plain}

\newtheorem{corollary}[theorem]{Corollary}

\newtheorem{example}[theorem]{Example}

\newtheorem{lemma}[theorem]{Lemma}

\newtheorem{proposition}[theorem]{Proposition}
\newtheorem{remark}[theorem]{Remark}

\numberwithin{equation}{section}

\def\ext#1#2#3{\mathcal{Q}(#1,#2,#3)}
\def\coc#1#2{\mathrm{C}(#1,#2)}
\def\zmap#1#2{\mathrm{Map}_0(#1,#2)}
\def\hom#1#2{\mathrm{Hom}(#1,#2)}
\def\cob#1#2{\mathrm{B}(#1,#2)}
\def\coh#1#2{\mathrm{H}(#1,#2)}
\def\inv#1{\mathrm{Inv}(#1)}
\def\invs#1{\mathrm{Inv}^*(#1)}
\def\invsc#1{\mathrm{Inv}^*_c(#1)}

\def\act#1#2#3{{}^{(#1,#2)}#3}
\def\conj#1#2{{}^{#1}#2}

\def\ker#1{\mathrm{Ker}\,#1}
\def\im#1{\mathrm{Im}\,#1}
\def\aut#1{\mathrm{Aut}(#1)}

\def\gf#1{\mathrm{GF}(#1)}

\def\nilp#1{\mathcal{N}(#1)}

\title{Enumeration of nilpotent loops via cohomology}

\author[Daly]{Daniel Daly}

\author[Vojt\v{e}chovsk\'y]{Petr Vojt\v{e}chovsk\'y}

\email{ddaly@math.du.edu, petr@math.du.edu}

\address{Department of Mathematics, University of Denver, 2360 S Gaylord St,
Denver, CO 80208, U.S.A.}

\begin{document}

\begin{abstract}
The isomorphism problem for centrally nilpotent loops can be tackled by methods
of cohomology. We develop tools based on cohomology and linear algebra that
either lend themselves to direct count of the isomorphism classes (notably in
the case of nilpotent loops of order $2q$, $q$ a prime), or lead to efficient
classification computer programs. This allows us to enumerate all nilpotent
loops of order less than $24$.
\end{abstract}



\maketitle

\section{Introduction}\label{Sc:Introduction}

A nonempty set $Q$ equipped with a binary operation $\cdot$ is a \emph{loop} if
it possesses a neutral element $1$ satisfying $1\cdot x = x\cdot 1 = x$ for
every $x\in Q$, and if for every $x\in Q$ the mappings $Q\to Q$, $y\mapsto
x\cdot y$ and $Q\to Q$, $y\mapsto y\cdot x$ are bijections of $Q$. From now on
we will abbreviate $x\cdot y$ as $xy$.

Note that multiplication tables of finite loops are precisely normalized latin
squares, and that groups are precisely associative loops.

The \emph{center} $Z(Q)$ of a loop $Q$ consists of all elements $x\in Q$ such
that
\begin{displaymath}
    xy = yx,\quad (xy)z = x(yz),\quad (yx)z = y(xz),\quad (yz)x = y(zx)
\end{displaymath}
for every $y$, $z\in Q$. \emph{Normal} subloops are kernels of loop
homomorphisms. The center $Z(Q)$ is a normal subloop of $Q$. The \emph{upper
central series} $Z_0(Q)\le Z_1(Q)\le \cdots$ is defined by
\begin{displaymath}
    Z_0(Q)=1,\quad Q/Z_{i+1}(Q) = Z(Q/Z_i(Q)).
\end{displaymath}
If there is $n\ge 0$ such that $Z_{n-1}(Q)<Z_n(Q)=Q$, we say that $Q$ is
\emph{(centrally) nilpotent of class $n$}.

The goal of this paper is to initiate the classification of small nilpotent
loops up to isomorphism, where by small we mean either that the order $|Q|$ of
$Q$ is a small integer, or that the prime factorization of $|Q|$ involves few
primes.

Here is a summary of the paper, with $A=(A,+)$ a finite abelian group and
$F=(F,\cdot)$ a finite loop throughout.

\S\ref{Sc:Cocycles}. Central extensions of $A$ by $F$ are in one-to-one
correspondence with (normalized) cocycles $\theta:F\times F\to A$. Let
$\ext{F}{A}{\theta}$ be the central extension of $A$ by $F$ via $\theta$. If
$\theta-\mu$ is a coboundary then $\ext{F}{A}{\theta}\cong \ext{F}{A}{\mu}$,
that is, the two loops are isomorphic.

\S\ref{Sc:Action}. The group $\aut{F,A}=\aut{F}\times\aut{A}$ acts on the
cocycles by
\begin{displaymath}
    (\alpha,\beta):\theta\mapsto\act{\alpha}{\beta}{\theta},\quad
    \act{\alpha}{\beta}{\theta}:(x,y)\mapsto
\beta\theta(\alpha^{-1}x,\alpha^{-1}y).
\end{displaymath}
For every $(\alpha,\beta)\in\aut{F,A}$ we have
$\ext{F}{A}{\theta}\cong\ext{F}{A}{\act{\alpha}{\beta}{\theta}}$.

Fix a cocycle $\theta$, and let us write $\theta\sim\mu$ if there is
$(\alpha,\beta)\in\aut{F,A}$ such that $\act{\alpha}{\beta}{\theta}-\mu$ is a
coboundary. If $\theta\sim\mu$, we have
$\ext{F}{A}{\theta}\cong\ext{F}{A}{\mu}$. If the converse is true for every
$\mu$, we say that $\theta$ is \emph{separable}. We describe several situations
in which all cocycles are separable.

\S\ref{Sc:Main}. If all cocycles are separable, the isomorphism problem for
central extensions reduces to the study of the equivalence classes of $\sim$.

For $(\alpha,\beta)\in\aut{F,A}$, let
\begin{displaymath}
    \inv{\alpha,\beta}=\{\theta;\theta-\act{\alpha}{\beta}{\theta}\text{ is a
    coboundary}\},
\end{displaymath}
and for $H\subseteq\aut{F,A}$, let
\begin{displaymath}
    \inv{H} = \bigcap_{(\alpha,\beta)\in H}\inv{\alpha,\beta}.
\end{displaymath}
Then $\inv{H}$ is a subgroup of cocycles, and $\inv{H}=\inv{\langle H\rangle}$,
where $\langle H\rangle$ is the subgroup of $\aut{F,A}$ generated by $H$.

For $H\le\aut{F,A}$, let
\begin{displaymath}
    \invs{H}=\inv{H}\setminus\bigcup_{H<K\le\aut{F,A}}\inv{K}.
\end{displaymath}
When $\theta\in\invs{H}$, the $\sim$-equivalence class $[\theta]_\sim$ of
$\theta$ is a union of precisely $[\aut{F,A}:H]$ cosets of coboundaries. It is
not necessarily true that $[\theta]_\sim$ is contained in $\invs{H}$, however,
it is contained in
\begin{displaymath}
    \invsc{H} = \bigcup_K \invs{K},
\end{displaymath}
where the union is taken over all subgroups $K$ of $\aut{F,A}$ conjugate to
$H$. Moreover, $|\invsc{H}| = |\invs{H}|\cdot[\aut{F,A}:N_{\aut{F,A}}(H)]$,
where $N_G(H)$ is the normalizer of $H$ in $G$.

Hence, if every cocycle is separable, we can enumerate all central extensions
of $A$ by $F$ up to isomorphism as soon as we know $|\invs{H}|$ for every
$H\le\aut{F,A}$, cf. Theorem \ref{Th:Main}.

\S\ref{Sc:Comp}. For $H$, $K\le \aut{F,A}$, we have $\inv{H}\cap\inv{K} =
\inv{\langle H\cup K\rangle}$. Hence $|\invs{K}|$ can be deduced from the
cardinalities of the subgroups $\inv{H}$ via the principle of inclusion and
exclusion based on the subgroup lattice of $\aut{F,A}$.

In turn, to find $|\inv{H}|$, it suffices to determine the cardinalities of
$\inv{\alpha,\beta}$ for every $(\alpha,\beta)\in H$, and the way these
subgroups intersect. When $A$ is a prime field, the action
$\theta\mapsto\act{\alpha}{\beta}{\theta}$ can be seen as a matrix operator on
the vector space of cocycles, and its preimage of coboundaries is
$\inv{\alpha,\beta}$. It is therefore not difficult to find
$\inv{\alpha,\beta}$ by means of (computer) linear algebra even for rather
large prime fields $A$ and loops $F$.

\S\ref{Sc:pq}. When $A=\mathbb Z_p$, $F=\mathbb Z_q$ and $p\ne q$ are primes,
the dimension of $\inv{\alpha,\beta}$ can be found without the assistance of a
computer, cf. Theorem \ref{Th:pq}.

\S\ref{Sc:2q}. Since every cocycle is separable when $p=2$ and $q$ is odd,
Theorems \ref{Th:Main} and \ref{Th:pq} give a formula for the number of
nilpotent loops of order $2q$, up to isomorphism, cf. Theorem \ref{Th:2q}. The
asymptotic growth of the number of nilpotent loops of order $2q$ is determined
in Theorem \ref{Th:Asymptotic2q}.

\S\ref{Sc:Inseparable}. Every central subloop contains $A=\mathbb Z_p$ for some
prime $p$. Not every choice of $A$ and $F$ results in separable cocycles, but
we can work around this problem when $A$ and $F$ are small by excluding the
subset $W(F,A)=\{\theta;\;Z(\ext{F}{A}{\theta})>A\}$, because all remaining
cocycles will be separable. When $W(F,A)$ is small, the isomorphism problem for
$\{\ext{F}{A}\theta;\theta\in W(F,A)\}$ can be tackled by a direct isomorphism
check, using the GAP package LOOPS.

\S\ref{Sc:Enumeration}. This allows us to enumerate all nilpotent loops of
order $n$ less than $24$ up to isomorphism, cf. Table \ref{Tb:Enumeration}. The
computational difficulties are nontrivial, notably for $n=16$ and $n=20$. We
accompany Table \ref{Tb:Enumeration} by a short narrative describing the
difficulties and how they were overcome.

There are $2,623,755$ nilpotent loops $F$ of order $12$, which is why the case
$n=24$ is out of reach of the methods developed here.

\S\ref{Sc:Remarks}. In order not to distract from the exposition, we have
collected references to related work and ideas at the end of the paper.

\section{Central extensions, cocycles and coboundaries}\label{Sc:Cocycles}

We say that a loop $Q$ is a \emph{central extension of $A$ by $F$} if $A\le
Z(Q)$ and $Q/A\cong F$.

A mapping $\theta:F\times F\to A$ is a \emph{normalized cocycle} (or
\emph{cocycle}) if it satisfies
\begin{equation}\label{Eq:LoopCocycle}
    \theta(1,x) = \theta(x,1) = 0 \text{ for every $x\in F$.}
\end{equation}
For a cocycle $\theta:F\times F\to A$, define $\ext{F}{A}{\theta}$ on $F\times
A$ by
\begin{equation}\label{Eq:CentralExtension}
    (x,a)(y,b) = (xy,a+b+\theta(x,y)).
\end{equation}

The following characterization of central loop extensions is well known, and is
in complete analogy with the associative case:

\begin{theorem} The loop $Q$ is a central extension of $A$ by $F$ if and only if there is a
cocycle $\theta:F\times F\to A$ such that $Q\cong\ext{F}{A}{\theta}$.
\end{theorem}

The cocycles $F\times F\to A$ form an abelian group $\coc{F}{A}$ with respect
to addition
\begin{displaymath}
    (\theta+\mu)(x,y) = \theta(x,y)+\mu(x,y).
\end{displaymath}
When $A$ is a field, $\coc{F}{A}$ is a vector space over $A$ with scalar
multiplication
\begin{displaymath}
    (c\theta)(x,y) = c\cdot \theta(x,y).
\end{displaymath}

Let
\begin{align*}
    \zmap{F}{A} &= \{\tau:F\to A;\;\tau(1)=0\},\\
    \hom{F}{A} &= \{\tau:F\to A;\;\tau\text{ is a homomorphism of loops}\},
\end{align*}
and observe:

\begin{lemma}
The mapping $\widehat{\phantom{x}}:\zmap{F}{A}\to \coc{F}{A}$, $\tau\mapsto
\widehat{\tau}$ defined by
\begin{displaymath}
    \widehat{\tau}(x,y) = \tau(xy)-\tau(x)-\tau(y)
\end{displaymath}
is a homomorphism of groups with kernel $\hom{F}{A}$.
\end{lemma}

The image $\cob{F}{A} = \widehat{\coc{F}{A}}\cong \zmap{F}{A}/\hom{F}{A}$ is a
subgroup (subspace) of $\coc{F}{A}$, and its elements are referred to as
\emph{coboundaries}.

When $A$ is a field, the vector space $\zmap{F}{A}$ has basis $\{\tau_c;\;c\in
F\setminus\{1\}\}$, where
\begin{equation}\label{Eq:Tau}
    \tau_c:F\to A,\quad
    \tau_c(x) = \left\{\begin{array}{ll}
    1,&\text{ if $x=c$,}\\
    0,&\text{ otherwise.}
    \end{array}\right.
\end{equation}
Hence the vector space $\cob{F}{A}$ is generated by $\{\widehat{\tau_c};\;c\in
F\setminus\{1\}\}$. Observe that for $x$, $y\in F\setminus\{1\}$ we have
\begin{equation}\label{Eq:TauHat}
    \widehat{\tau_c}(x,y) = \left\{\begin{array}{ll}
        1,&\text{ if $xy=c$},\\
        -1,&\text{ if $x=c$ or $y=c$ but not $x=y$},\\
        -2,&\text{ if $x=y=c$},\\
        0,&\text{ otherwise.}
    \end{array}\right.
\end{equation}

Coboundaries play a prominent role in classifications due to this simple
observation:

\begin{lemma}\label{Lm:Coboundaries}
Let $\widehat{\tau}\in\cob{F}{A}$. Then $f:\ext{F}{A}{\theta} \to
\ext{F}{A}{\theta+\widehat{\tau}}$ defined by
\begin{displaymath}
    f(x,a) = (x,a+\tau(x))
\end{displaymath}
is an isomorphism of loops.
\end{lemma}

The converse of Lemma \ref{Lm:Coboundaries} does not hold, making the
classification of loops up to isomorphism nontrivial even in highly structured
subvarieties, such as groups. Nevertheless it is clear that it suffices to
consider cocycles modulo coboundaries, and we therefore define the
\emph{(second) cohomology} $\coh{F}{A} = \coc{F}{A}/\cob{F}{A}$.

\section{The action of the automorphism groups and separability}\label{Sc:Action}

Let $\aut{F,A}=\aut{F}\times\aut{A}$. The group $\aut{F,A}$ acts on
$\coc{F}{A}$ via
\begin{displaymath}
    \theta\mapsto\act{\alpha}{\beta}{\theta},\quad
    \act{\alpha}{\beta}{\theta}:(x,y)\mapsto \beta\theta(\alpha^{-1}x,\alpha^{-1}y).
\end{displaymath}
Indeed, we have $\act{\alpha\gamma}{\beta\delta}{\theta} =
\act{\alpha}{\beta}{(\act{\gamma}{\delta}{\theta})}$, and
$\act{\alpha}{\beta}{(\theta+\mu)} = \act{\alpha}{\beta}{\theta} +
\act{\alpha}{\beta}{\mu}$. Since
\begin{displaymath}
    \act{\alpha}{\beta}{\widehat{\tau}} = \widehat{\beta\tau\alpha^{-1}},
\end{displaymath}
the action of $\aut{F,A}$ on $\coc{F}{A}$ induces an action on $\cob{F}{A}$ and
on $\coh{F}{A}$. Moreover:

\begin{lemma}\label{Lm:ActionToIso}
Let $(\alpha,\beta)\in\aut{F,A}$. Then
$f:\ext{F}{A}{\theta}\to\ext{F}{A}{\act{\alpha}{\beta}{\theta}}$ defined by
\begin{displaymath}
    f(x,a)=(\alpha x, \beta a)
\end{displaymath}
is an isomorphism of loops.
\end{lemma}
\begin{proof}
Let $\cdot$ be the multiplication in $\ext{F}{A}{\theta}$ and $*$ the
multiplication in $\ext{F}{A}{\act{\alpha}{\beta}{\theta}}$. Then
\begin{align*}
    f((x,a)\cdot(y,b)) &= f(xy,a+b+\theta(x,y)) = (\alpha(xy),\beta(a+b+\theta(x,y)))\\
    &= (\alpha(x)\alpha(y), \beta(a)+\beta(b) + \beta\theta(\alpha^{-1}\alpha x,\alpha^{-1}\alpha y))\\
    &= (\alpha(x)\alpha(y), \beta(a)+\beta(b) + \act{\alpha}{\beta}{\theta}(\alpha x,\alpha y))\\
    &= (\alpha x,\beta a)*(\alpha y,\beta b) = f(x,a)*f(y,b).
\end{align*}
\end{proof}

As in \S \ref{Sc:Introduction}, write $\theta\sim\mu$ if there is
$(\alpha,\beta)\in\aut{F,A}$ such that
$\act{\alpha}{\beta}{\theta}-\mu\in\cob{F}{A}$. Then $\sim$ is an equivalence
relation on $\coc{F}{A}$, and the equivalence class of $\theta$ is
\begin{displaymath}
    [\theta]_\sim = \bigcup_{(\alpha,\beta)\in\aut{F,A}} (\act{\alpha}{\beta}{\theta} +
    \cob{F}{A}).
\end{displaymath}
By Lemmas \ref{Lm:Coboundaries} and \ref{Lm:ActionToIso}, if $\theta\sim\mu$
then $\ext{F}{A}{\theta}\cong\ext{F}{A}{\mu}$. We say that $\theta$ is
\emph{separable} if the converse is also true, that is, if
$\ext{F}{A}{\theta}\cong\ext{F}{A}{\mu}$ if and only if $\theta\sim\mu$.

We remark that there exists an inseparable cocycle already in $\coc{\mathbb
Z_6}{\mathbb Z_2}$. In the rest of this section we describe situations that
guarantee separability.

\begin{proposition}\label{Pr:Admissible}
Let $Q=\ext{F}{A}{\theta}$. If $\aut{Q}$ acts transitively on $\{K\le Z(Q);\;
K\cong A$, $Q/K\cong F\}$ then $\theta$ is separable.
\end{proposition}
\begin{proof}
Let $Q=\ext{F}{A}{\theta}$, and let $f:Q\to\ext{F}{A}{\mu}$ be an isomorphism.
Let $K=f^{-1}(1\times A)$. By our assumption, there is $g\in\aut{Q}$ such that
$g(1\times A) = K$. Then $fg:Q\to\ext{F}{A}{\mu}$ is an isomorphism mapping
$1\times A$ onto itself. We can therefore assume without loss of generality
that already $f$ has this property.

Denote by $\cdot$ the multiplication in $Q$ and by $*$ the multiplication in
$\ext{F}{A}{\mu}$. Define $\beta:A\to A$ by $(1,\beta(a)) = f(1,a)$. Then
\begin{displaymath}
    (1,\beta(a+b)){=}f(1,a+b){=}f((1,a)\cdot(1,b)){=}f(1,a)*f(1,b)
    {=}(1,\beta a)*(1,\beta b){=}(1,\beta a + \beta b),
\end{displaymath}
which means that $\beta\in\aut{A}$.

Define $\tau:F\to A$ and $\alpha:F\to F$ by $f(x,0) = (\alpha x, \tau x)$.
Since $f(1,0)=(1,0)$, we have $\tau\in\zmap{F}{A}$. Moreover, calculating
modulo $A$ in both loops, we have
\begin{displaymath}
    (\alpha(xy),0) \equiv f(xy,0){\equiv}f((x,0)\cdot(y,0))
    {\equiv}f(x,0)*f(y,0){\equiv}(\alpha x,0)*(\alpha y,0){\equiv}(\alpha(x)\alpha(y),0),
\end{displaymath}
and $\alpha\in\aut{F}$ follows.

The isomorphism $f$ satisfies
\begin{displaymath}
    f(x,a) = f((1,a)\cdot(x,0)) = f(1,a)*f(x,0) = (1,\beta a)*(\alpha x, \tau
    x) = (\alpha x, \beta a + \tau x).
\end{displaymath}
If is therefore the composition of the isomorphism $(x,a)\mapsto
(x,a+\beta^{-1}\tau x)$ of Lemma \ref{Lm:Coboundaries} (with $\beta^{-1}\tau$
in place of $\tau$) and of the isomorphism $(x,a)\mapsto (\alpha x,\beta a)$ of
Lemma \ref{Lm:ActionToIso}. This means that $\mu =
\act{\alpha}{\beta}{(\theta+\widehat{\beta^{-1}\tau})}$, so $\mu \in
\act{\alpha}{\beta}{\theta}+\cob{F}{A}$, $\theta\sim\mu$.
\end{proof}

We now investigate separability in abelian groups. The next two results can be
proved in many ways from the Fundamental Theorem of Finitely Generated Abelian
Groups, which we use without warning.

\begin{lemma}\label{Lm:CyclicComplement} Let $p$ be a prime, and let
\begin{equation}\label{Eq:ADecomposition}
    A = \mathbb Z_{p^{e_1}}\times\cdots\times\mathbb Z_{p^{e_n}}
\end{equation}
be an abelian $p$-group, where $e_1\le\cdots\le e_n$. Let $x\in A$ be an
element of order $p$. Then there exists a unique integer $e_j$ such that: there
is a complemented cyclic subgroup $B\le A$ satisfying $x\in B$ and
$|B|=p^{e_j}$. Moreover,
\begin{displaymath}
    A/\langle x\rangle\cong \mathbb Z_{p^{f_1}}\times\cdots\times \mathbb
    Z_{p^{f_n}},
\end{displaymath}
where $f_i=e_i$ for every $i\ne j$, and $f_j=e_j-1$.
\end{lemma}
\begin{proof}
Every element $x\in A$ of order $p$ is of the form
\begin{displaymath}
    x = (x_1p^{e_1-1},\dots,x_np^{e_n-1}),
\end{displaymath}
where $x_i\in\{0,\dots,p-1\}$ for every $1\le i\le n$, and where $x_i\ne 0$ for
some $1\le i\le n$. Let $j$ be the least integer such that $x_j\ne 0$. Consider
the element
\begin{displaymath}
    y = \frac{x}{p^{e_j-1}} =
    (0,\dots,0,x_j,x_{j+1}p^{e_{j+1}-e_j},\dots,x_np^{e_n-e_j}).
\end{displaymath}
Then $B=\langle y\rangle$ contains $x$, $|B|=p^{e_j}$, and
\begin{displaymath}
    C = \mathbb Z_{p^{e_1}}\times\cdots\times \mathbb Z_{p^{e_{j-1}}}\times 0
    \times \mathbb Z_{p^{e_{j+1}}}\times\cdots\times \mathbb Z_{p^{e_n}}
\end{displaymath}
is a complement of $B$ in $A$ (that is, $B\cap C=0$ and $\langle B\cup C\rangle
= A$).
\end{proof}

\begin{proposition}\label{Pr:AdmissibleAbelian}
Let $A$ be a finite abelian group. For a prime $p$ dividing $|A|$ and for a
finite abelian group $F$ of order $|A|/p$, let
\begin{displaymath}
    X(p,F) = \{x\in A;\;|x|=p\text{ and }A/\langle x\rangle\cong F\}.
\end{displaymath}
Then the sets $X(p,F)$ that are nonempty are precisely the orbits of the action
of $\aut{A}$ on $A$.
\end{proposition}
\begin{proof}
For a prime $p$, let $A_p$ be the $p$-primary component of $A$. Then
$A=A_{p_1}\times\cdots\times A_{p_m}$, for some distinct primes $p_1$, $\dots$,
$p_m$, and $\aut{A} = \aut{A_{p_1}}\times\cdots\times\aut{A_{p_n}}$. (For a
detailed proof, see \cite[Lemma 2.1]{HR}.) We can therefore assume that $A=A_p$
is a $p$-group.

It is obvious that every orbit of $\aut{A}$ is contained in one of the sets
$X(p,F)$. It therefore suffices to prove that if $x$, $y\in X(p,F)$ then there
is $\varphi\in\aut{A}$ such that $\varphi(x)=y$.

Let $A$ be as in \eqref{Eq:ADecomposition}. If $A$ is cyclic of order $p^{e_1}$
then $A/\langle x\rangle\cong \mathbb Z_{p^{e_1-1}}$, and we can assume that
$x=ap^{e_1-1}$, $y = bp^{e_1-1}$, where $1\le a$, $b\le p-1$. The automorphism
of $A$ determined by $1\mapsto b/a$ (modulo $p$) then maps $a$ to $b$ and hence
$x$ to $y$.

Assume that $n>1$. Let $B_x$, $B_y$ be the complemented cyclic subgroups $B$
obtained by Lemma \ref{Lm:CyclicComplement} for $x$, $y$, respectively. Then
$|B_x|=|B_y|$ since $A/\langle x\rangle \cong A/\langle y\rangle$, and hence
the integer $e_j$ determined by Lemma \ref{Lm:CyclicComplement} is the same for
$x$ and $y$. We can in fact assume that already $j$ is the same. Furthermore,
we can assume that the isomorphism from
\begin{displaymath}
    A/\langle x\rangle \cong \mathbb Z_{p^{e_1}}\times\cdots\times\mathbb Z_{p^{e_{j-1}}}
    \times B_x/\langle x\rangle \times \mathbb Z_{p^{e_{j+1}}}\times\cdots
    \times \mathbb Z_{p^{e_n}}
\end{displaymath}
to
\begin{displaymath}
    A/\langle y\rangle \cong \mathbb Z_{p^{e_1}}\times\cdots\times\mathbb Z_{p^{e_{j-1}}}
    \times B_y/\langle y\rangle \times \mathbb Z_{p^{e_{j+1}}}\times\cdots
    \times \mathbb Z_{p^{e_n}}
\end{displaymath}
is componentwise, and maps $B_x/\langle x\rangle$ to $B_y/\langle y\rangle$. We
can then extend $B_x/\langle x\rangle\to B_y/\langle y\rangle$ to an
isomorphism $B_x\to B_y$ while sending $x$ to $y$ by the case $n=1$, and hence
obtain the desired automorphism of $A$.
\end{proof}

\begin{corollary}\label{Cr:AdmissibleAbelian}
Let $Q=\ext{F}{A}{\theta}$ be an abelian group, $A=\mathbb Z_p$, $p$ a prime.
Then $\theta$ is separable.
\end{corollary}
\begin{proof}
Combine Propositions \ref{Pr:Admissible} and \ref{Pr:AdmissibleAbelian}.
\end{proof}

Finally, we show that all cocycles are separable in ``small'' situations.

\begin{lemma}\label{Lm:Index2}
There is no loop $Q$ with $[Q:Z(Q)]=2$.
\end{lemma}
\begin{proof}
Assume, for a contradiction, that $|Q/Z(Q)|=2$, and let $a\in Q\setminus Z(Q)$.
Then every element of $Q$ can be written as $a^iz$, where $i\in\{0,1\}$ and
$z\in Z(Q)$. For every $i$, $j$, $k\in\{0,1\}$ and $z_1$, $z_2$, $z_3\in Z(Q)$
we have $a^iz_1\cdot (a^jz_2\cdot a^kz_3) = a^i(a^ja^k)\cdot z_1z_2z_3$, and
similarly, $(a^iz_1\cdot a^jz_2)\cdot a^kz_3 = (a^ia^j)a^k\cdot z_1z_2z_3$. The
two expressions are equal if any of $i$, $j$, $k$ vanishes. So it remains to
discuss the case $i=j=k=1$. But then $a(aa) = (aa)a$, because $a^2\in Z(Q)$.
Hence $Q$ is a group. It is well known that if $Q$ is a group and $Q/Z(Q)$ is
cyclic then $Q=Z(Q)$, a contradiction.
\end{proof}

Lemma \ref{Lm:Index2} cannot be improved: for every odd prime $p$ there is a
nonassociative loop $Q$ such that $|Q/Z(Q)|=p$, cf. Theorem \ref{Th:2q}.

\begin{lemma}\label{Lm:Admissible}
Let $Q=\ext{F}{A}{\theta}$, $A=\mathbb Z_p$, $p$ a prime. Assume
further that one of the following conditions is satisfied:
\begin{enumerate}
\item[(i)] $|Q|=p$,
\item[(ii)] $|Q|=pq$, where $q$ is a prime,
\item[(iii)] $[Q:Z(Q)]\le 2$,
\item[(iv)] $|Q|<12$.
\end{enumerate}
Then $\theta$ is separable.
\end{lemma}
\begin{proof}
When (i) or (iii) hold then $Q$ is an abelian group by Lemma \ref{Lm:Index2},
and so $\theta$ is separable by Corollary \ref{Cr:AdmissibleAbelian}.

Assume that (ii) holds. If $Z(Q)>A$ then $Z(Q)=Q$ and we are done by Corollary
\ref{Cr:AdmissibleAbelian}. Else $Z(Q)=A$ and $\theta$ is separable by
Proposition \ref{Pr:Admissible}, for trivial reasons.

To finish (iv), it remains to discuss the case $|Q|=8$. If $Z(Q)=A$, $\theta$
is separable by Proposition \ref{Pr:Admissible}. If $Z(Q)>A$ then $Z(Q)=Q$ by
Lemma \ref{Lm:Index2}, and we are done by Corollary \ref{Cr:AdmissibleAbelian}.
\end{proof}

\section{The invariant subspaces}\label{Sc:Main}

For $(\alpha,\beta)\in\aut{F,A}$, let
\begin{equation}\label{Eq:Inv}
    \inv{\alpha,\beta} = \{\theta\in\coc{F}{A};\;
    \theta - \act{\alpha}{\beta}{\theta}\in\cob{F}{A}\}.
\end{equation}
For $\emptyset\ne H\subseteq \aut{F,A}$, let
\begin{equation}\label{Eq:InvH}
    \inv{H} = \bigcap_{(\alpha,\beta)\in H} \inv{\alpha,\beta}.
\end{equation}

\begin{lemma}\label{Lm:InvH}
Let $\emptyset\ne H\subseteq \aut{F,A}$. Then $\inv{H} = \inv{\langle
H\rangle}$.
\end{lemma}
\begin{proof}
Assume that $\theta\in\inv{\alpha,\beta}\cap\inv{\gamma,\delta}$. Then
$\theta-\act{\alpha}{\beta}{\theta}\in\cob{F}{A}$ and
$\theta-\act{\gamma}{\delta}{\theta}\in\cob{F}{A}$. The second equation is
equivalent to $\act{\alpha}{\beta}{\theta} -
\act{\alpha}{\beta}{(\act{\gamma}{\delta}{\theta})}\in\cob{F}{A}$. Adding this
to the first equation yields $\theta -
\act{\alpha}{\beta}{(\act{\gamma}{\delta}{\theta})} = \theta -
\act{\alpha\gamma}{\beta\delta}{\theta}\in\cob{F}{A}$.
\end{proof}

\begin{corollary}\label{Cr:InvHK}
Let $H$, $K\le \aut{F,A}$. Then $\inv{H}\cap\inv{K} = \inv{\langle H\cup
K\rangle}$.
\end{corollary}

For $\alpha$, $\gamma\in \aut{F}$ and $\beta$, $\delta\in\aut{A}$, let
$\conj{\gamma}{\alpha} = \gamma\alpha\gamma^{-1}$, $\conj{\delta}{\beta} =
\delta\beta\delta^{-1}$.

\begin{lemma}\label{Lm:Conj}
Let $(\alpha,\beta)$, $(\gamma,\delta)\in \aut{F,A}$. Then
$\theta\in\inv{\alpha,\beta}$ if and only if $\act{\gamma}{\delta}{\theta} \in
\inv{\conj{\gamma}{\alpha},\conj{\delta}{\beta}}$.
\end{lemma}
\begin{proof}
The following conditions are equivalent:
\begin{align*}
    \act{\gamma}{\delta}{\theta}&\in \inv{\conj{\gamma}{\alpha},\conj{\delta}{\beta}},\\
    \act{\gamma}{\delta}{\theta} - \act{\conj{\gamma}{\alpha}}{\conj{\delta}{\beta}}{(\act{\gamma}{\delta}{\theta})}&\in \cob{F}{A},\\
    \act{\gamma}{\delta}{\theta} - \act{\gamma\alpha}{\delta\beta}{\theta}&\in \cob{F}{A},\\
    \act{\gamma}{\delta}{(\theta - \act{\alpha}{\beta}{\theta})} &\in \cob{F}{A},\\
    \theta - \act{\alpha}{\beta}{\theta}& \in \cob{F}{A},\\
    \theta &\in \inv{\alpha,\beta}.
\end{align*}
\end{proof}

For $H\le\aut{F,A}$, let
\begin{align*}
    \invs{H} &= \{\theta\in\coc{F}{A};\;\theta\in\inv{\alpha,\beta}\text{ if and only if }(\alpha,\beta)\in H\},\\
    \invsc{H} &= \bigcup_{(\alpha,\beta)\in \aut{F,A}} \invs{\conj{(\alpha,\beta)}{H}}.
\end{align*}
As we are going to see, the cardinality of the equivalence class
$[\theta]_\sim$ can be easily calculated for $\theta\in\invs{H}$, provided
$\theta$ is separable.

If $G$ is a group and $H\le G$, let $N_G(H)=\{a\in G;\;\conj{a}{H}=H\}$ be the
normalizer of $H$ in $G$.

\begin{lemma}\label{Lm:ConjSize}
Let $H\le G = \aut{F,A}$. Then $|\invsc{H}| = |\invs{H}|\cdot [G:N_G(H)]$.
\end{lemma}
\begin{proof}
Since $\conj{a}{H} = \conj{b}{H}$ if and only if $a^{-1}b\in N_G(H)$, there are
precisely $[G:N_G(H)]$ subgroups $K$ of $G$ conjugate to $H$.

Assume that $K\ne H$ are conjugate, $K=\conj{(\alpha,\beta)}{H}$. The mapping
$f:\coc{F}{A}\to \coc{F}{A}$, $\theta\mapsto\act{\alpha}{\beta}{\theta}$ is a
bijection. By Lemma \ref{Lm:Conj}, $f(\inv{H}) = \inv{K}$, and $f(\invs{H}) =
\invs{K}$, proving $|\invs{H}| = |\invs{K}|$. Since $K\ne H$, we have
$\invs{H}\cap\invs{K}=\emptyset$ by definition.
\end{proof}

For a group $G$, denote by $\mathrm{Sub}_c(G)$ a set of subgroups of $G$ such
that for every $H\le G$ there is precisely one $K\in\mathrm{Sub}_c(G)$ such
that $K$ is conjugate to $H$.

\begin{theorem}\label{Th:Main}
Let $F$ be a loop and $A$ an abelian group. Assume that $\theta$ is separable
for every $\theta\in\coc{F}{A}$. Let $G = \aut{F,A}$. Then there are
\begin{equation}\label{Eq:Main}
    \sum_{H\in\mathrm{Sub}_c(G)}
        \frac{ |\invsc{H}| }{ |\cob{F}{A}|\cdot [G:H] }
    = \sum_{H\in\mathrm{Sub}_c(G)}
        \frac{ |\invs{H}|}{ |\cob{F}{A}|\cdot [N_G(H):H] }
\end{equation}
central extensions of $A$ by $F$, up to isomorphism.
\end{theorem}
\begin{proof}
By Lemma \ref{Lm:InvH},
\begin{displaymath}
    \coc{F}{A} = \bigcup_{H\le G}\invs{H}
    = \bigcup_{H\in\mathrm{Sub}_c(G)}\invsc{H},
\end{displaymath}
where the unions are disjoint. Let $\theta\in\invsc{H}$, for some $H\le G$. Let
$X=\{\mu\in\coc{F}{A};\;\ext{F}{A}{\mu}\cong\ext{F}{A}{\theta}\}$. Since
$\theta$ is separable, we have
\begin{displaymath}
    X = [\theta]_\sim = \bigcup_{(\alpha,\beta)\in G}
    (\act{\alpha}{\beta}{\theta}+\cob{F}{A}) \subseteq \invsc{H},
\end{displaymath}
where the first equality follows by separability of $\theta$, and the inclusion
from Lemma \ref{Lm:Conj}.

Let $K$ be the unique conjugate of $H$ such that $\theta\in\invs{K}$. We have
$\act{\alpha}{\beta}{\theta} - \act{\gamma}{\delta}{\theta}\in\cob{F}{A}$ if
and only if $\theta -
\act{\alpha^{-1}\gamma}{\beta^{-1}\delta}{\theta}\in\cob{F}{A}$, which holds if
and only if $\theta\in\inv{\alpha^{-1}\gamma,\beta^{-1}\delta}$. Since
$\theta\in\invs{K}$, we see that $\act{\alpha}{\beta}{\theta} -
\act{\gamma}{\delta}{\theta}\in\cob{F}{A}$ holds if and only if
$(\alpha^{-1}\gamma,\beta^{-1}\delta)\in K$, or
$(\alpha,\beta)K=(\gamma,\delta)K$. Hence $[\theta]_\sim$ is a union of $[G:K]
= [G:H]$ cosets of $\cob{F}{A}$. We have established the first sum of
\eqref{Eq:Main}. The second sum then follows from Lemma \ref{Lm:ConjSize}.
\end{proof}

\section{Calculating the subspaces $\inv{\alpha,\beta}$ by
computer}\label{Sc:Comp}

Assume throughout this section that $A=\mathbb Z_p$, where $p$ is a prime. Then
$\coc{F}{A}$, $\cob{F}{A}$ and $\coh{F}{A}$ are vector spaces over $\gf{p}$.

For $(\alpha,\beta)\in\aut{F,A}$ let $R=R(\alpha,\beta)$, $S=S(\alpha,\beta)$
be the linear operators $\coc{F}{A}\to\coc{F}{A}$ defined by
\begin{align*}
    R(\alpha,\beta)\theta &= \act{\alpha}{\beta}{\theta},\\
    S(\alpha,\beta)\theta &= \theta - \act{\alpha}{\beta}{\theta}.
\end{align*}
Hence $R(\alpha,\beta)$ is invertible, and $S(\alpha,\beta) =
I-R(\alpha,\beta)$, where $I:\coc{F}{A}\to\coc{F}{A}$ is the identity operator.

As $\beta\in\aut{\mathbb Z_p}$ is a scalar multiplication by $\beta(1)$, let us
identify $\beta$ with $\beta(1)$. Then $R(\alpha,\beta)$ is a matrix operator
with rows and columns labeled by pairs of nonidentity elements of $F$, where
the only nonzero coefficient in row $(x,y)$ is $-\beta$ in column $(\alpha^{-1}
x,\alpha^{-1} y)$.

By definition of $\inv{\alpha,\beta}$ and $S(\alpha,\beta)$, we have
\begin{displaymath}
    \inv{\alpha,\beta} = \{\theta\in\coc{F}{A};\;S(\alpha,\beta)\theta
    \in\cob{F}{A}\} = S(\alpha,\beta)^{-1}\cob{F}{A}.
\end{displaymath}
In order to calculate $\inv{\alpha,\beta}$, we can proceed as follows:
\begin{enumerate}
\item[$\bullet$] calculate the subspace $\cob{F}{A}$ as the span of
    $\{\widehat{\tau_c};\;1\ne c\in F\}$,
\item[$\bullet$] calculate the kernel $\ker{S(\alpha,\beta)}$ and image
    $\im{S(\alpha,\beta)}$ as usual,
\item[$\bullet$] find a basis $\mathcal B$ of the subspace
    $\cob{F}{A}\cap\im{S(\alpha,\beta)}$,
\item[$\bullet$] for $b\in\mathcal B$, find a particular solution
    $\theta_b$ to the system $S(\alpha,\beta)\theta_b = b$,
\item[$\bullet$] then $\inv{\alpha,\beta} = \ker{S(\alpha,\beta)} \oplus
    \langle \theta_b;\;b\in\mathcal B\rangle$.
\end{enumerate}
In particular, with $S=S(\alpha,\beta)$, we have
\begin{align*}
    \dim\inv{\alpha,\beta} &= \dim\ker{S}
    + \dim( \im{S}\cap\cob{F}{A})\\
    &=\dim\ker{S} + \dim\im{S} + \dim\cob{F}{A}
    - \dim(\im{S}+\cob{F}{A})\\
    &=(|F|-1)^2 + \dim\cob{F}{A} - \dim(\im{S}+\cob{F}{A}).
\end{align*}

Using a computer, it is therefore not difficult to find $\inv{\alpha,\beta}$
and its dimension even for rather large loops $A=\mathbb Z_p$ and $F$. See
\S\ref{Sc:Enumeration} for more details.

\begin{remark}
If it is preferable to operate modulo coboundaries, note that
$S(\alpha,\beta)(\theta + \widehat{\tau}) = S(\alpha,\beta)\theta +
\widehat{\tau - \beta\tau\alpha^{-1}}$, and view $S(\alpha,\beta)$ as a linear
operator $S(\alpha,\beta):\coh{F}{A}\to\coh{F}{A}$ defined by
\begin{displaymath}
    S(\alpha,\beta)(\theta + \cob{F}{A}) = (\theta -
    \act{\alpha}{\beta}{\theta}) + \cob{F}{A}.
\end{displaymath}
Then $\inv{\alpha,\beta}/\cob{F}{A} = \ker{S(\alpha,\beta)}$.
\end{remark}

\section{The subspaces $\inv{\alpha,\beta}$ for $A=\mathbb Z_p$, $F=\mathbb
Z_q$}\label{Sc:pq}

If $H\le K\le\aut{F,A}$, we have $\inv{K}\le\inv{H}$. Hence the subgroups
$\inv{H}$ will be incident in accordance with the upside down subgroup lattice
of $\aut{F,A}$, except that some edges in the lattice can collapse, i.e., it
can happen that $\inv{H} = \inv{K}$ although $H<K$:

\begin{example}\label{Ex:Collapse}
Let $A=\mathbb Z_2$, $F=\mathbb Z_2\times\mathbb Z_2$. Then $\aut{F}\cong S_3$.
Let $H$ be the subgroup of $\aut{F}$ generated by a $3$-cycle. Then it turns
out that $\mathrm{Inv}(H) = \mathrm{Inv}(\aut{F})$.
\end{example}

Such a collapse has no impact on the formula \eqref{Eq:Main} of Theorem
\ref{Th:Main}, since only subgroups $H$ with $\mathrm{Inv}^*(H)\ne\emptyset$
contribute to it.

We proceed to determine $\dim\inv{\alpha,\beta}$.

In addition to the operators $R(\alpha,\beta)$ and $S(\alpha,\beta)$ on
$\coc{F}{A}$, define $T(\alpha,\beta)$ by
\begin{displaymath}
    T(\alpha,\beta)\theta = \theta + R(\alpha,\beta)\theta + \cdots + R(\alpha,\beta)^{k-1}\theta,
\end{displaymath}
where $k=|\alpha|$.

\begin{lemma}\label{Lm:Operators}
Let $R$, $S$, $T$ be operators on a finite-dimensional vector space $V$ such
that $R^k=I$, $S=I-R$, $T=I+R+\cdots+R^{k-1}$. Then $\im{T}\le\ker{S}$ and
$\im{S}\le\ker{T}$. If $\im{T}=\ker{S}$ then $\ker{T}=\im{S}$.
\end{lemma}
\begin{proof}
We have $TS = (I+R+\cdots +R^{k-1})(I-R) =  I-R^k= 0$ and $ST =
(I-R)(I+R+\cdots +R^{k-1}) = 0$, which shows $\im{T}\le\ker{S}$,
$\im{S}\le\ker{T}$.

Assume that $\im{T}=\ker{S}$. By the Fundamental Homomorphism Theorem,
\begin{displaymath}
    \dim\im{T}+\dim\ker{T} = \dim V = \dim\im{S}+\dim\ker{S} =
    \dim\im{S}+\dim\im{T},
\end{displaymath}
so $\dim\ker{T}=\dim\im{S}$. Since $\im{S}\le\ker{T}$, we conclude that
$\im{S}=\ker{T}$.
\end{proof}

\begin{lemma}\label{Lm:Null}
Let $p$, $q$ be primes, $A=\mathbb Z_p$, $F=\mathbb Z_q$, $\alpha\in\aut{F}$,
$\beta\in\aut{A}$.
\begin{enumerate}
\item[(i)] If $|\beta|$ does not divide $|\alpha|$ then $S(\alpha,\beta)$
    is invertible.
\item[(ii)] If $|\beta|$ divides $|\alpha|$ then $\ker{S(\alpha,\beta)} =
    \im{T(\alpha,\beta)}$ and $\dim\ker{S(\alpha,\beta)} =
    (q-1)^2/|\alpha|$.
\end{enumerate}
\end{lemma}
\begin{proof}
Let $F^*=F\setminus\{0\}$, $k=|\alpha|$. The automorphism $\alpha$ acts on
$F^*\times F^*$ via $(x,y)^\alpha = (\alpha^{-1}x, \alpha^{-1}y)$. Every
$\alpha$-orbit has size $k$. Let $t=(q-1)^2/k$, and let $\mathcal O_1$,
$\dots$, $\mathcal O_t$ be all the distinct $\alpha$-orbits on $F^*\times F^*$.

Let $R=R(\alpha,\beta)$, $S=S(\alpha,\beta)$, $T=T(\alpha,\beta)$. Throughout
the proof, let $\theta\in\ker{S}$, i.e.,
\begin{equation}\label{Eq:InKernel}
    \theta(x,y) = \beta\theta(\alpha^{-1}x,\alpha^{-1}y)
\end{equation}
for every $x$, $y\in F^*$. For every $1\le i\le t$, let $(x_i,y_i)\in\mathcal
O_i$. Define $\theta_i\in\coc{F}{A}$ by
\begin{displaymath}
    \theta_i(x,y) = \left\{\begin{array}{ll}
        \theta(x_i,y_i),&\text{ if $(x,y)=(x_i,y_i)$,}\\
        0,&\text{ otherwise.}
    \end{array}\right.
\end{displaymath}
Then
\begin{displaymath}
    \theta = \sum_{i=1}^t T\theta_i = T(\sum_{i=1}^t \theta_i),
\end{displaymath}
thus $\ker{S}\le\im{T}$.

The condition \eqref{Eq:InKernel} implies $\theta(x,y) = \beta^k\theta(x,y)$
for every $x$, $y\in F^*$. If $|\beta|$ does not divide $|\alpha|$, we have
$\beta^k\ne 1$, and therefore $\theta=0$, proving $\ker{S}=0$.

Assume that $|\beta|$ divides $|\alpha|$. Then $R^k=I$, and $\im{T}\le\ker{S}$
by Lemma \ref{Lm:Operators}. Thus $\im{T}=\ker{S}$ and $\ker{T}=\im{S}$. Since
$\theta$ is determined by the values $\theta(x_i,y_i)$, for $1\le i\le t$, and
since these values can be arbitrary, we see that $\dim{\ker{S}} = t$.
\end{proof}

\begin{lemma}\label{Lm:D}
Let $p$, $q$ be distinct primes, $A=\mathbb Z_p$, $F=\mathbb Z_q$,
$\alpha\in\aut{F}$, $\beta\in\aut{A}$, and assume that $|\beta|$ divides
$|\alpha|$. Then
\begin{displaymath}
    \dim{(\ker{T(\alpha,\beta)}\cap \cob{F}{A})} = (q-1)\left(1-\frac{1}{|\alpha|}\right).
\end{displaymath}
\end{lemma}
\begin{proof}
The set $\{\widehat{\tau_c};\;c\in F^*\}$ is linearly independent thanks to
$p\ne q$. Let
\begin{displaymath}
    \widehat{\tau} = \sum_{c\in F^*}\lambda_c\widehat{\tau_c},
\end{displaymath}
for some $\lambda_c\in A$. An inspection of \eqref{Eq:TauHat} reveals that
\begin{equation}\label{Eq:TauHatAlpha}
    R(\alpha,\beta)\widehat{\tau_c} = \beta\widehat{\tau_{\alpha c}}.
\end{equation}
Thus $\widehat{\tau}$ belongs to $\ker{T(\alpha,\beta)}\cap\cob{F}{A}$ if and
only if
\begin{equation}\label{Eq:System2}
    \sum_c \lambda_c\widehat{\tau_c} +
    \beta \sum_c \lambda_c\widehat{\tau_{\alpha c}} +
    \cdots
    + \beta^{k-1}\sum_c\lambda_c\widehat{\tau_{\alpha^{(k-1)}c}} = 0.
\end{equation}
The coefficient of $\widehat{\tau_c}$ in \eqref{Eq:System2} is
$\lambda_c+\beta\lambda_{\alpha^{-1} c} + \cdots +
\beta^{k-1}\lambda_{\alpha^{-(k-1)}c}$, so the system \eqref{Eq:System2} can be
rewritten in terms of the coefficients $\lambda_c$ as
\begin{equation}\label{Eq:System3}
    \lambda_c+\beta\lambda_{\alpha^{-1} c} + \cdots +
    \beta^{k-1}\lambda_{\alpha^{-(k-1)}c} = 0, \text{ for $c\in F^*$.}
\end{equation}
For any $1\le i\le k$, the equation for $c$ is a scalar multiple of the
equation for $\alpha^ic$. On the other hand, each equation involves scalars $c$
from only one orbit of $\alpha$. Hence \eqref{Eq:System3} reduces to a system
of $(q-1)/|\alpha|$ linearly independent equations in variables $\lambda_c$,
$c\in F^*$. It follows that the subspace of homogeneous solutions has dimension
$(q-1)(1-1/|\alpha|)$.
\end{proof}

\begin{theorem}\label{Th:pq}
Let $p\ne q$ be primes, $A=\mathbb Z_p$, $F=\mathbb Z_q$, $\alpha\in\aut{F}$,
$\beta\in\aut{A}$. Then
\begin{displaymath}
    \inv{\alpha,\beta} = \ker{S(\alpha,\beta)} + \cob{F}{A}.
\end{displaymath}
Moreover,
\begin{displaymath}
    \dim(\inv{\alpha,   \beta}) = \left\{\begin{array}{ll}
        q-1,&\text{ if $|\beta|$ does not divide $|\alpha|$},\\
        (q-1) + (q-1)(q-2)/|\alpha|,&\text{ otherwise.}
    \end{array}\right.
\end{displaymath}
Thus
\begin{displaymath}
    \dim(\inv{\alpha,\beta}/\cob{F}{A}) = \left\{\begin{array}{ll}
        0,&\text{ if $|\beta|$ does not divide $|\alpha|$},\\
        (q-1)(q-2)/|\alpha|,&\text{ otherwise.}
    \end{array}\right.
\end{displaymath}
\end{theorem}
\begin{proof}
Let $R=R(\alpha,\beta)$, $S=S(\alpha,\beta)$, $T=T(\alpha,\beta)$ and
$B=\cob{F}{A}$. Assume that $|\beta|$ does not divide $|\alpha|$. Then $S$ is
invertible by Lemma \ref{Lm:Null}, so $\inv{\alpha,\beta} = S^{-1}B = B =
\ker{S} + B$, and we have $\dim\inv{\alpha,\beta} = \dim B = q-1$ thanks to
$p\ne q$.

Now assume that $|\beta|$ divides $|\alpha|$. By Lemmas \ref{Lm:Operators},
\ref{Lm:Null} and \ref{Lm:D}, we have $\im{T} = \ker{S}$, $\dim(\im{S}\cap B) =
(q-1)(1-1/|\alpha|)$, and $\dim\ker{S} = (q-1)^2/|\alpha|$, so
\begin{align*}
    \dim\inv{\alpha,\beta} &= \dim\ker{S} + \dim(\im{S}\cap B)\\
    &= (q-1)^2/|\alpha| + (q-1)(1-1/|\alpha|)\\
    &= (q-1) + (q-1)(q-2)/|\alpha|.
\end{align*}
It remains to show that $\inv{\alpha,\beta} = \ker{S}+B$.

Let $k=|\alpha|$. The coboundaries $\{\widehat{\tau_c};\;c\in F^*\}$ are
linearly independent thanks to $p\ne q$. For $1\le i\le m = (q-1)/k$, let $c_i$
be a representative of the coset $c_i\langle\alpha\rangle$ in $F^*$, and assume
that $\bigcup_{i=1}^m c_i\langle\alpha\rangle = F^*$. By
\eqref{Eq:TauHatAlpha}, the set $\{R^\ell\widehat{\tau_{c_i}};\;0\le\ell\le
k-2$, $1\le i\le m\}$ is linearly independent, and so is its $S$-image
$\{R^\ell\widehat{\tau_{c_i}}-R^{\ell+1}\widehat{\tau_{c_i}};\;0\le\ell\le
k-2$, $1\le i\le m\}\subseteq B$. This shows that $\dim(S(B)\cap B)\ge
(q-1)(1-1/k)$. On the other hand, $\dim(\im{S}\cap B) = (q-1)(1-1/k)$ by Lemma
\ref{Lm:D}. Thus $\im{S}\cap B = S(B)\cap B = S(B)$. But this means that
$\inv{\alpha,\beta} = S^{-1}B$ is equal to $\ker{S}+B$.
\end{proof}

\section{Nilpotent loops of order $2q$, $q$ a prime}\label{Sc:2q}

For $n\ge 1$, let $\nilp{n}$ be the number of nilpotent loops of order $n$ up
to isomorphism. In this section we find a formula for $\nilp{2q}$, where $q$ is
a prime, and describe the asymptotic behavior of $\nilp{2q}$ as $q\to\infty$.

Loops of order $4$ are associative, and, up to isomorphism, there are $2$
nilpotent groups of order $4$, namely $\mathbb Z_4$ and $\mathbb
Z_2\times\mathbb Z_2$.

\begin{theorem}\label{Th:2q}
Let $q$ be an odd prime. For a positive integer $d$, let
\begin{displaymath}
    \mathrm{Pred}(d) = \{d';\;1\le d'<d,\,d/d'\text{ is a prime}\}
\end{displaymath}
be the set of all maximal proper divisors of $d$. Then the number of nilpotent
loops of order $2q$ up to isomorphism is
\begin{equation}\label{Eq:2q}
    \nilp{2q} = \sum_{d\text{ divides }q-1} \frac{1}{d}\left(
    2^{(q-2)d} + \sum_{\emptyset\ne D\subseteq \mathrm{Pred}(d)}
    (-1)^{|D|}\cdot 2^{(q-2)\gcd{D}}\right).
\end{equation}
\end{theorem}
\begin{proof}
By Lemma \ref{Lm:Index2}, the only central extension of $\mathbb Z_q$ by
$\mathbb Z_2$ is the cyclic group $\mathbb Z_{2q}$. Since this group can also
be obtained as a central extension of $\mathbb Z_2$ by $\mathbb Z_q$, we can
set $A=\mathbb Z_2$, $F=\mathbb Z_q$. Then by Lemma \ref{Lm:Admissible}, every
$\theta\in\coc{F}{A}$ is separable, so Theorem \ref{Th:Main} applies.

We have $\aut{F,A} = \aut{F}=\langle\alpha\rangle\cong\mathbb Z_{q-1}$. The
subgroup structure of $\aut{F}$ is therefore transparent: for every divisor $d$
of $q-1$ there is a unique subgroup $H_d=\langle\alpha^d\rangle$ of order
$(q-1)/d$, and if $d$, $d'$ are two divisors of $q-1$ then $\langle H_d\cup
H_{d'}\rangle = H_{\gcd(d,d')}$.

By Theorem \ref{Th:pq},
\begin{displaymath}
    \dim\inv{H_d} = \dim\inv{\alpha^d} = (q-1)+(q-1)(q-2)/((q-1)/d)
    =(q-1)+(q-2)d,
\end{displaymath}
so $\dim(\inv{H_d}/\cob{F}{A}) = (q-2)d$.

Note that $H_d$ is a maximal subgroup of $H_{d'}$ if and only if
$d'\in\mathrm{Pred}(d)$. For $\emptyset\ne D\subseteq\mathrm{Pred}(d)$, we have
$\langle H_{d'};\;d'\in D\rangle = H_{\gcd{D}}$, and so
\begin{displaymath}
    \bigcap_{d'\in D}\inv{H_{d'}} = \inv{H_{\gcd{D}}}
\end{displaymath}
by Corollary \ref{Cr:InvHK}. Then
\begin{displaymath}
    |\invs{H_d}| = |\inv{H_d}| + \sum_{\emptyset\ne D\subseteq\mathrm{Pred}(d)}
    (-1)^{|D|}\cdot |\inv{H_{\gcd{D}}}|
\end{displaymath}
by the principle of inclusion and exclusion.

As $\aut{F}$ is abelian, $[N_{\aut{F}}(H_d):H_d] = [\aut{F}:H_d] = d$. The
formula \eqref{Eq:2q} then follows by Theorem \ref{Th:Main}.
\end{proof}

\begin{example} To illustrate \eqref{Eq:2q}, let us determine $\mathcal N(14)= \mathcal N(2\cdot
7)$. The divisors of $q-1=6$ are $6$, $3$, $2$, $1$. Hence
\begin{equation}\label{Eq:14}
    \mathcal N(14) = (2^{5\cdot 6} - 2^{5\cdot 3} - 2^{5\cdot 2} + 2^{5\cdot 1})/6
    + (2^{5\cdot 3}-2^{5\cdot 1})/3 + (2^{5\cdot 2} - 2^{5\cdot 1})/2
    + 2^{5\cdot 1}/1 = 178,962,784.
\end{equation}
\end{example}

\begin{table}
\caption{The number $\nilp{2q}$ of nilpotent loops of order $2q$, $q$ a prime,
up to isomorphism.}\label{Tb:2q}
\begin{tiny}
\begin{displaymath}
    \begin{array}{l|c}
    2q&\nilp{2q}\\
    \hline
    4&2\\
    6&3\\
    10&1,044\\
    14&178,962,784\\
    22&123,794,003,928,541,545,927,226,368\\
    26&453,709,822,561,251,284,623,981,727,533,724,162,048\\
    34&110,427,941,548,649,020,598,956,093,796,432,407,322,294,493,291,283,427,083,203,517,192,617,984
    \end{array}
\end{displaymath}
\end{tiny}
\end{table}

Table \ref{Tb:2q} lists the number of nilpotent loops of order $2q$ up to
isomorphism for small primes $q$. (It is by no means difficult to evaluate
\eqref{Eq:2q} for larger primes, say up to $q\le 100$, but the decimal
expansion of $\mathcal N(2q)$ becomes too long to display neatly in a table.)

Here is the asymptotic growth of $\mathcal N(2q)$:

\begin{theorem}\label{Th:Asymptotic2q}
Let $q$ be an odd prime. Then the number of nilpotent loops of order $2q$ up to
isomorphism is approximately $2^{(q-2)(q-1)}/(q-1)$. More precisely,
\begin{displaymath}
    \lim_{q\text{ prime, }q\to\infty} \mathcal N(2q)\cdot \frac{q-1}{2^{(q-2)(q-1)}} =
    1.
\end{displaymath}
\end{theorem}
\begin{proof}
We prove the assertion by a simple estimate. To illustrate the main idea, note
that \eqref{Eq:14} can be rewritten as
\begin{displaymath}
    2^{30}/6 + 2^{15}(1/3-1/6) + 2^{10}(1/2-1/6) + 2^5(1-1/2-1/3+1/6).
\end{displaymath}
Thus, upon rewriting \eqref{Eq:2q} in a similar fashion, there will be no more
than $q-1$ summands, each of the form
\begin{equation}\label{Eq:Summand}
    2^{(q-2)d'}(1/{d_1}\pm 1/{d_2} \pm \cdots \pm 1/{d_m}).
\end{equation}
A reciprocal $1/d$ appears in \eqref{Eq:Summand} if and only if there is a
divisor $d$ of $q-1$ and $D\subseteq\mathrm{Pred}(d)$ such that $\gcd{D} = d'$.
Now, for every divisor $d$ of $q-1$ there is at most one subset
$D\subseteq\mathrm{Pred}(d)$ such that  $\gcd{D}=d'$ (because if
$D=\{e_1,\dots,e_n\}$, $d/e_i=p_i$ is a prime, then $\gcd{D} = d/(p_1\cdots
p_n)$ uniquely determines $D$). Hence the number of reciprocals in
\eqref{Eq:Summand} cannot exceed $q-1$. Finally, the largest proper divisor of
$q-1$ is $(q-1)/2$. Altogether,
\begin{displaymath}
    \frac{2^{(q-2)(q-1)}}{q-1} - (q-1)2^{(q-2)(q-1)/2}(q-1) \le N(2q) \le
    \frac{2^{(q-2)(q-1)}}{q-1} + (q-1)2^{(q-2)(q-1)/2}(q-1),
\end{displaymath}
thus
\begin{displaymath}
    1 - \frac{(q-1)^3}{2^{(q-2)(q-1)/2}} \le N(2q)\cdot
    \frac{(q-1)}{2^{(q-2)(q-1)}} \le 1 + \frac{(q-1)^3}{2^{(q-2)(q-1)/2}},
\end{displaymath}
and the result follows by the Squeeze Theorem.
\end{proof}

\section{Inseparable cocycles}\label{Sc:Inseparable}

Let $A=\mathbb Z_p$, $F$ be as usual. The easiest (but slow) way to deal with
inseparable cocycles $\theta\in\coc{F}{A}$ is to treat separately the subset
\begin{displaymath}
    W(F,A) =
    \{\theta\in\coc{F}{A};\;Z(\ext{F}{A}{\theta})>A\}\subseteq\coc{F}{A}.
\end{displaymath}
We will refer to elements of $W(F,A)$ informally as \emph{large center
cocycles}. Note that the adjective ``large'' is relative to $A$. The subset
$W(F,A)$ can be determined computationally as follows:

Let $Q=\ext{F}{A}{\theta}$. The element $(x,a)$ belongs to $Z(Q)$ if and only
if $\{(x,b);\;b\in A\}\subseteq Z(Q)$, which happens if and only if $x\in Z(F)$
and $\theta$ satisfies
\begin{align*}
    \theta(x,y) &= \theta(y,x),\\
    \theta(x,y) + \theta(xy,z) &= \theta(y,z) + \theta(x,yz),\\
    \theta(y,x) + \theta(yx,z) &= \theta(x,z) + \theta(y,xz),\\
    \theta(y,z) + \theta(yz,x) &= \theta(z,x) + \theta(y,zx)
\end{align*}
for all $y$, $z\in F$. The first condition ensures that $(x,a)$ commutes with
all elements of $Q$, and the last three conditions ensure that $(x,a)$
associates with all elements of $Q$, no matter in which position $(x,a)$
happens to be in the associative law. (Note that the last condition is a
consequence of the first three.)

Hence for every $1\ne x\in Z(F)$ we can solve the above linear equations and
obtain the subspace $W_x(F,A)\le\coc{F}{A}$ such that $\theta\in W_x(F,A)$ if
and only if $(x,A)\subseteq Z(\ext{F}{A}{\theta})$. Then
\begin{displaymath}
    W(F,A)=\bigcup_{1\ne x\in Z(F)} W_x(F,A),
\end{displaymath}
and this subset can be determined by the principle of inclusion and exclusions
on the subspaces $W_x(F,A)$, $1\ne x\in Z(F)$.

Importantly, every cocycle $\theta\in\coc{F}{A}\setminus W(F,A)$ is separable,
since then $\ext{F}{A}{\theta}$ possesses a unique central subloop of the
cardinality $|A|$, namely $A$.

When $A$, $F$ are small, we can complete the isomorphism problem by first
constructing the loops $\ext{F}{A}{\theta}$ for all $\theta\in
W(F,A)/\cob{F}{A}$ and then sorting them up to isomorphism by standard
algorithms of loop theory. Since these algorithms are slow, dealing with large
center cocycles is the main obstacle in pushing the enumeration of nilpotent
loops past order $n=23$.

\section{Enumeration of nilpotent loops of order less than $24$}\label{Sc:Enumeration}

\begin{table}
\caption{The number of nilpotent loops up to
isomorphism.}\label{Tb:Enumeration}
\begin{displaymath}
    \begin{array}{c|c|c|c|c}
    n&A&F&\#Q,\,Z(Q)>A&\#Q\\
    \hline
    \hline
    4&\mathbb Z_2&\mathbb Z_2&2&2\\
    \hline
    6&\mathbb Z_2&\mathbb Z_3&1&3\\
    \hline
    8&\mathbb Z_2&\mathbb Z_4&2&80\\
    8&\mathbb Z_2&\mathbb Z_2\times\mathbb Z_2&2&60\\
    8&\mathbb Z_2&4&3&139\\
    \hline
    9&\mathbb Z_3&\mathbb Z_3&2&10\\
    \hline
    10&\mathbb Z_2&\mathbb Z_5&1&1,044\\
    \hline
    12&\mathbb Z_2&\mathbb Z_6&6&1,049,560\\
    12&\mathbb Z_2&L_{6,2}&4&1,048,576\\
    12&\mathbb Z_2&L_{6,3}&4&525,312\\
    12&\mathbb Z_2&6&11&2,623,485\\
    12&\mathbb Z_3&\mathbb Z_4&1&196\\
    12&\mathbb Z_3&\mathbb Z_2\times\mathbb Z_2&1&76\\
    12&\mathbb Z_3&4&2&272\\
    12&&&&2,623,755\\
    \hline
    14&\mathbb Z_2&\mathbb Z_7&1&178,962,784\\
    \hline
    15&\mathbb Z_3&\mathbb Z_5&1&66,626\\
    15&\mathbb Z_5&\mathbb Z_3&1&5\\
    15&&&&66,630\\
    \hline
    16&\mathbb Z_2&8&9,284&466,409,543,467,341\\
    \hline
    18&\mathbb Z_2&9&34&157,625,987,549,892,128\\
    18&\mathbb Z_3&\mathbb Z_6&10&2,615,147,350\\
    18&\mathbb Z_3&L_{6,2}&14&5,230,176,602\\
    18&\mathbb Z_3&L_{6,3}&10&2,615,147,350\\
    18&\mathbb Z_3&6&34&10,460,471,302\\
    18&&&&157,625,998,010,363,396\\
    \hline
    20&\mathbb Z_2&10&2,798,987&4,836,883,870,081,433,134,082,379\\
    20&\mathbb Z_5&\mathbb Z_4&1&1,985\\
    20&\mathbb Z_5&\mathbb Z_2\times\mathbb Z_2&1&685\\
    20&\mathbb Z_5&4&2&2,670\\
    20&&&&4,836,883,870,081,433,134,085,047\\
    \hline
    21&\mathbb Z_3&\mathbb Z_7&1&17,157,596,742,628\\
    21&\mathbb Z_7&\mathbb Z_3&1&6\\
    21&&&&17,157,596,742,633\\
    \hline
    22&\mathbb Z_2&\mathbb Z_{11}&1&123,794,003,928,541,545,927,226,368
    \end{array}
\end{displaymath}
\end{table}

The results are summarized in Table \ref{Tb:Enumeration}. A typical line of the
table can be read as follows: ``$\#Q$'' is the number of nilpotent loops (up to
isomorphism) of order $n$ that are central extensions of the cyclic group
$A=\mathbb Z_p$ by the nilpotent loop $F$ of order $n/p$. If only the order of
$F$ is given, $F$ is any of the nilpotent loops of order $n/p$. If no
information about $A$ and $F$ is given, any pair $(A,F)$ with $A=\mathbb Z_p$,
$F$ nilpotent of order $n/p$ can be used. Finally, ``$\#Q,\,Z(Q)>A$'' is the
number of nilpotent loops with center larger than $A$. Since this makes sense
only when $A$ is specified, we omit ``$\#Q,\,Z(Q)>A$'' in the other cases.

By Lemma \ref{Lm:Admissible}, we can apply the formula \eqref{Eq:Main} safely
until we reach order $n=12$.

For every prime $p$ there is a unique nilpotent loop of order $p$ up to
isomorphism, namely the cyclic group $\mathbb Z_p$.

The number of nilpotent loops of order $2q$, $q$ a prime, is determined by
Theorem \ref{Th:2q}. Note, however, that the theorem does not produce the
loops. Since we need all nilpotent loops of order $6$ and $10$ explicitly in
order to compute the number of nilpotent loops of order $12$, $18$ and $20$, we
must obtain the nilpotent loops of order $6$, $10$ by other means (a direct
isomorphism check on $\coh{F}{A}$ will do).

In accordance with Theorem \ref{Th:2q}, there are $3$ nilpotent loops of order
$6$. Beside the cyclic group of order $6$, the other two loops are
\begin{displaymath}
    \begin{array}{c|cccccc}
    L_{6,2}&1&2&3&4&5&6\\
    \hline
    1&1& 2& 3& 4& 5& 6 \\  2&2& 1& 4& 3& 6& 5 \\  3&3& 4& 5& 6& 1& 2\\
    4&4& 3& 6& 5& 2& 1 \\  5&5& 6& 2& 1& 3& 4 \\  6&6& 5& 1& 2& 4& 3.
    \end{array},
    \quad\quad
    \begin{array}{c|cccccc}
    L_{6,3}&1&2&3&4&5&6\\
    \hline
    1&1&2&3&4&5&6\\ 2&2&1&4&3&6&5\\ 3&3&4&5&6&1&2\\
    4&4&3&6&5&2&1\\ 5&5&6&1&2&4&3\\ 6&6&5&2&1&3&4
    \end{array}.
\end{displaymath}

\subsection{n=8}

Case $A=\mathbb Z_2$, $F=\mathbb Z_4$. We have $\aut{A}=1$,
$\aut{F}=\langle\alpha\rangle\cong\mathbb Z_2$, and $\dim\hom{F}{A}=1$,
$\dim\cob{F}{A}=2$, $\dim\coc{F}{A}=9$. Computer yields $\dim\inv{\alpha} = 7$.
Hence \eqref{Eq:Main} shows that there are
\begin{displaymath}
    \frac{2^7}{2^2} + \frac{2^9-2^7}{2^2\cdot 2} = 80
\end{displaymath}
central extensions of $\mathbb Z_2$ by $\mathbb Z_4$, up to isomorphism.

Case $A=\mathbb Z_2$, $F=\mathbb Z_2\times\mathbb Z_2$. We have $\aut{A}=1$,
$\aut{F}=\langle \sigma,\rho\rangle\cong S_3$, where $|\sigma|=2$, $|\rho|=3$.
Furthermore, $\dim\hom{F}{A}=2$, $\dim\cob{F}{A}=1$ and $\dim\coc{F}{A}=9$. The
three subspaces $\inv{\sigma}$, $\inv{\sigma\rho}$, $\inv{\sigma\rho^2}$ have
dimension $6$, and any two of them intersect precisely in $\inv{\rho}$ (see
Example \ref{Ex:Collapse}), which has dimension $3$. By \eqref{Eq:Main}, there
are
\begin{displaymath}
    \frac{2^3}{2} + 3\cdot\frac{2^6-2^3}{2\cdot 3} + \frac{2^9 - 3\cdot 2^6 +
    2\cdot 2^3}{2\cdot 6} = 60
\end{displaymath}
central extensions of $\mathbb Z_2$ by $\mathbb Z_2\times\mathbb Z_2$.

In order to pinpoint the number of nilpotent loops of order $8$, we must
determine which loops are obtained both as central extensions of $\mathbb Z_2$
by $\mathbb Z_4$ and of $\mathbb Z_2$ by $\mathbb Z_2\times\mathbb Z_2$. First
of all, $\mathbb Z_2\times\mathbb Z_4$ is such a loop. Assume that $Q$ is
another such loop. Then $|Z(Q)|>2$ and hence $Q$ is an abelian group by Lemma
\ref{Lm:Index2}. Now, $Q\ne\mathbb Z_8$ since every factor $Z_8/\langle
x\rangle$ by an involution is isomorphic to $\mathbb Z_4$. Finally,
$Q\ne\mathbb Z_2\times\mathbb Z_2\times\mathbb Z_2$ since every factor by an
involution is of exponent $2$. We conclude that there are $80+60-1 = 139$
nilpotent loops of order $8$.

\subsection{n=9}

We have $A=\mathbb Z_3$, $F=\mathbb Z_3$, $\dim\hom{F}{A}=1$,
$\dim\cob{F}{A}=1$ and $\dim\coc{F}{A}=4$. Also, $\aut{A}=\langle
\beta\rangle\cong \mathbb Z_2$, $\aut{F} = \langle \alpha\rangle \cong\mathbb
Z_2$.

By computer, $\inv{\beta}=\cob{F}{A}$ has dimension $1$, $\dim \inv{\alpha}=2$,
$\dim\inv{\alpha\beta}=3$, and $\inv{\alpha\beta}\cap \inv{\alpha} =
\inv{\beta}$. Then \eqref{Eq:Main} gives
\begin{displaymath}
    \frac{3}{3} + \frac{3^3-3}{3\cdot 2} + \frac{3^2-3}{3\cdot 2} +
    \frac{3^4 - 3^3 - 3^2 + 3}{3\cdot 4} = 10
\end{displaymath}
nilpotent loops of order $9$.

\subsection{n=12}

For the first time we have to worry about separability, and hence we have to
calculate the subsets $W(F,A)$.

Case $A=\mathbb Z_2$, $F=\mathbb Z_6$. Let $\aut{F}=\langle
\alpha\rangle\cong\mathbb Z_2$. The subset $W(F,A)$ is in fact a subspace: Let
$x\in F$ be the unique involution and $y\in F$ an element of order $3$. If
$\theta\in W_y(F,A)$ then $Z(\ext{F}{A}{\theta})=\ext{F}{A}{\theta}$ by Lemma
\ref{Lm:Index2}. Thus $W_y(F,A)\subseteq W_x(F,A) = W(F,A)$.

Computer calculation yields $\dim W(F,A) = 7$, $\dim\cob{F}{A}=4$,
$\dim\inv{\alpha}= 15$, and $\dim (W(F,A)\cap\inv{\alpha}) = 6$. Thus there are
\begin{displaymath}
    \frac{2^{25} - 2^7 - (2^{15}-2^6)}{2^4\cdot 2} + \frac{2^{15}-2^6}{2^4} =
    1,049,594
\end{displaymath}
loops $Q$ with $|Q|=12$ and $Q/Z(Q)=\mathbb Z_6$. Among the $2^7/2^4 = 8$ loops
constructed from the large center cocycles, $6$ are nonisomorphic.

Case $A=\mathbb Z_2$, $F=L_{6,2}$. By computer, $\aut{F}=1$,
$\dim\cob{F}{A}=5$, $\dim W(F,A) = 7$. Thus there are
\begin{displaymath}
    \frac{2^{25}-2^7}{2^5} = 1,048,572
\end{displaymath}
loops $Q$ with $|Q|=12$ and $Q/Z(Q)=F$. The $2^7/2^5=4$ loops corresponding to
cocycles in $W(F,A)$ are pairwise nonisomorphic.

Case $A=\mathbb Z_2$, $F=L_{6,3}$. Then computer gives
$\aut{F}=\langle\alpha\rangle\cong\mathbb Z_2$, $\dim\cob{F}{A}=5$,
$\dim{W(F,A)}=7$, $\dim\inv{\alpha}=16$, and $W(F,A)\le\inv{\alpha}$. Thus
there are
\begin{displaymath}
    \frac{2^{25}-2^{16}}{2^5\cdot 2} + \frac{2^{16}-2^7}{2^5} = 525,308
\end{displaymath}
nilpotent loops $Q$ with $|Q|=12$ and $Q/Z(Q)=F$. The $2^7/2^5 = 4$ loops
corresponding to large center cocycles are pairwise nonisomorphic.

Among the $6+4+4$ loops with $|Z(Q)|>2$ found so far, $11$ are nonisomorphic.

Case $A=\mathbb Z_3$, $|F|=4$. If $Z(Q)>A$ then $[Q:Z(Q)]\le 2$, so all
cocycles in $\coc{F}{A}$ are separable by Lemma \ref{Lm:Admissible}. The
details are in Table \ref{Tb:Enumeration}.

If a nilpotent loop of order $12$ is a central extension of both $\mathbb Z_2$
and of $\mathbb Z_3$, it is an abelian group by Lemma \ref{Lm:Index2}, and
hence it is isomorphic to $\mathbb Z_2\times\mathbb Z_2\times \mathbb Z_3$ or
to $\mathbb Z_4\times\mathbb Z_3$. We have counted these two loops twice and
must take this into account.

\subsection{n=15}

Either $A=\mathbb Z_3$, $F=\mathbb Z_5$ or $A=\mathbb Z_5$, $F=\mathbb Z_3$. In
both cases, all cocycles are separable by Lemma \ref{Lm:Admissible}. Most
subspaces $\inv{H}$ can be determined by Theorem \ref{Th:pq}. The two cases
overlap only in $\mathbb Z_3\times\mathbb Z_5$.

\subsection{n=16}

This is a more difficult case due to the $139$ nilpotent loops $F_1$, $\dots$,
$F_{139}$ of order $8$.

Cases $A=\mathbb Z_2$, $F=F_i$. We calculate the subsets $W_i = W(F_i,A)$, and
treat separable cocycles outside $W_i$ as usual. (In one of the cases, the
automorphism group $\aut{F,A}=\aut{F}$ is the simple group of order $168$, the
largest automorphism group we had to deal with in the entire search.) We filter
the large center loops up to isomorphism.

We now need to filter the union of the $139$ sets of large center loops up to
isomorphism. This can be done efficiently as follows: Let
$Q=\ext{F}{A}{\theta}$ where $\theta\in W_i$. For every central involution $x$
of $Q$, calculate $Q/\langle x\rangle$ and determine its isomorphism type. If
$Q/\langle x\rangle$ is isomorphic to some $F_j$ with $j<i$, we have already
seen $Q$ and can discard it.

\subsection{n=18}

See Table \ref{Tb:Enumeration}.

\subsection{n=20}

This is the computationally most difficult case, due to the $1,044$ nilpotent
loops of order $10$. See \S\ref{Sc:Remarks} for more. The efficient filtering
of large center loops is crucial here. On the other hand, $1,008$ out of the
$1,044$ nilpotent loops of order $10$ have trivial automorphism groups.

\subsection{n=21}

This case is analogous to $n=15$.

\section{Related ideas and concluding remarks}\label{Sc:Remarks}

For an introduction to loop theory see Bruck \cite{Bruck} or Pflugfelder
\cite{Pflugfelder}.

The study of (central) extensions of groups by means of cocycles goes back to
Schreier \cite{Schreier}. The abstract cohomology theory for groups was
initiated by Eilenberg and MacLane in \cite{EM1}--\cite{EM3}, and it has grown
into a vast subject.

Eilenberg and MacLane were also the first to investigate cohomology of loops.
In \cite{EM4}, they imposed conditions on loop cocycles that mimic those of
group cocycles, and calculated some cohomology groups. A more natural theory
(by many measures) of loop cohomology has been developed in \cite{JL} by
Johnson and Leedham-Green. As in this paper, their third cohomology group
vanishes, since they impose no conditions on the (normalized) loop
$2$-cocycles.

We are not aware of any work on the classification of nilpotent loops \emph{per
se}. In the recent paper \cite{MMM}, McKay, Mynert and Myrvold enumerated all
loops of order $n\le 10$ up to isomorphism. We believe that all results in \S
\ref{Sc:Main}--\S \ref{Sc:Enumeration} are new.

This being said, the central notion of separable cocycles must have surely been
noticed before, but since it is of limited utility in group theory (where much
stronger structural results are available to attack the isomorphism problem of
central extensions), it has not been investigated in the more general setting
of loops. The experienced reader will recognize the mappings $S$ and $T$ of \S
\ref{Sc:pq} as the consecutive differentials in a free resolution of a cyclic
group, cf. \cite[Ch.\ 2]{Evens}.

The computational tools developed here are applicable to finitely based
varieties of loops, and can therefore be used to classify nilpotent loops of
small orders in such varieties. One merely has to start with the appropriate
space of cocycles (determined by a system of linear equations, just as in the
group case). The first author intends to undertake this classification for
loops of Bol-Moufang type, cf. \cite{PV}. The classification of all Moufang
loops of order $n\le 64$ and $n=81$ can be found in \cite{NV}. The
classification of Bol loops has been started in \cite{Moorhouse}. The LOOPS
\cite{LOOPS} package contains libraries of small loops in certain varieties,
including Bol and Moufang loops.

All calculations in this paper have been carried out in the GAP \cite{GAP}
package LOOPS. We wrote two mostly independent codes, and the calculations have
been done at least twice. The enumeration of nilpotent loops of order $20$ took
more than $90$ percent of the total calculation time, about $2$ days on a
single-processor Unix machine. Both codes and the multiplication tables of all
nilpotent loops of order $n\le 10$ can be downloaded at the second author's web
site \texttt{http://www.math.du.edu/\~{}petr}.

\bibliographystyle{plain}

\end{document}